\documentclass{amsart}

\usepackage{amsthm}
\usepackage[leqno]{amsmath}
\usepackage{latexsym,amsfonts,amssymb}
\usepackage[all]{xy} \SelectTips{eu}{} \SilentMatrices
\usepackage{hyperref}

\newcommand{\numberseries}{\mdseries}   

\newlength{\thmtopspace}                
\newlength{\thmbotspace}                
\newlength{\thmheadspace}               
\newlength{\thmindent}                  

\renewcommand{\subparagraph}{\vspace*{\thmbotspace}}

\newtheoremstyle{bfupright head,slanted body}
                {\thmtopspace}{\thmbotspace}
                {\slshape}{\thmindent}{\bfseries}{.}{\thmheadspace}
                {{\numberseries \thmnumber{(#2) }}\thmnote{#3}}

\newtheoremstyle{bfupright head,upright body}
                {\thmtopspace}{\thmbotspace}
                {\upshape}{\thmindent}{\bfseries}{.}{\thmheadspace}
                {{\numberseries \thmnumber{(#2) }}\thmnote{#3}}

\newtheoremstyle{bfit head,upright body}
                {\thmtopspace}{\thmbotspace}
                {\upshape}{\thmindent}{\upshape}{.}{\thmheadspace}
                {{\numberseries\thmnumber{(#2) }}
                {\bfseries\itshape\thmnote{\negthickspace#3}}}

\newtheoremstyle{it head,upright body}
                {\thmtopspace}{\thmbotspace}
                {\upshape}{\thmindent}{\upshape}{.}{\thmheadspace}
                {{\numberseries\thmnumber{(#2) }}
                {\itshape\thmnote{\negthickspace#3}}}

\newtheoremstyle{fixed bf head,slanted body}
                {\thmtopspace}{\thmbotspace}{\slshape}
                {\thmindent}{\bfseries}{.}{\thmheadspace}
                {{\numberseries \thmnumber{(#2) }}\thmname{#1}\thmnote{ (#3)}}

\newtheoremstyle{fixed bf head,upright body}
                {\thmtopspace}{\thmbotspace}{\upshape}
                {\thmindent}{\bfseries}{.}{\thmheadspace}
                {{\numberseries \thmnumber{(#2) }}\thmname{#1}\thmnote{ (#3)}}

\newtheoremstyle{fixed bfit head,upright body}
                {\thmtopspace}{\thmbotspace}{\upshape}
                {\thmindent}{\bfseries\itshape}{.}{\thmheadspace}
                {{\numberseries \thmnumber{(#2) }}\thmname{#1}\thmnote{ (#3)}}

\newtheoremstyle{sc head,small body}
                {\thmtopspace}{\thmbotspace}
                {\small\upshape}{\thmindent}{\scshape}{.}{\thmheadspace}
                {\thmname{#1}}

\newtheoremstyle{numbered paragraph}
                {\thmtopspace}{\thmbotspace}{\upshape}
                {\thmindent}{\upshape}{}{0pt}
                {{\numberseries \thmnumber{(#2) }}}

\newtheoremstyle{unnumbered paragraph}
                {\thmtopspace}{\thmbotspace}{\upshape}
                {\parindent}{\upshape}{}{0pt}

\theoremstyle{bfupright head,slanted body}
\newtheorem{res}{}[section]             \newtheorem*{res*}{}

\theoremstyle{bfit head,upright body}
                 \newtheorem*{com*}{}

\theoremstyle{bfupright head,upright body}
\newtheorem{bfhpg}[res]{}               \newtheorem*{bfhpg*}{}

\theoremstyle{it head,upright body}
               \newtheorem*{ithpg*}{}

\theoremstyle{sc head,small body}

\theoremstyle{fixed bf head,slanted body}
\newtheorem{thm}[res]{Theorem}          \newtheorem*{thm*}{Theorem}
\newtheorem{prp}[res]{Proposition}      \newtheorem*{prp*}{Proposition}
        \newtheorem*{cor*}{Corollary}
\newtheorem{lem}[res]{Lemma}            \newtheorem*{lem*}{Lemma}

\theoremstyle{fixed bf head,upright body}
\newtheorem{dfn}[res]{Definition}       \newtheorem*{dfn*}{Definition}
     \newtheorem*{con*}{Construction}
      \newtheorem*{obs*}{Observation}
\newtheorem{rmk}[res]{Remark}           \newtheorem*{rmk*}{Remark}
\newtheorem{exa}[res]{Example}          \newtheorem*{exa*}{Example}
         \newtheorem*{exe*}{Exercise}
\newtheorem{stp}[res]{Setup}            \newtheorem{stp*}{Setup}

\theoremstyle{numbered paragraph}

\theoremstyle{unnumbered paragraph}
\newtheorem{ipg*}{}

\newlength{\thmlistleft}        
\newlength{\thmlistright}       
\newlength{\thmlistpartopsep}   
\newlength{\thmlisttopsep}      
\newlength{\thmlistparsep}      
\newlength{\thmlistitemsep}     

\newcounter{eqc} 
  {\end{list}}%

\newcounter{prt}
\newenvironment{prt}{\begin{list}{\upshape (\alph{prt})}%
    {\usecounter{prt}%
      \setlength{\leftmargin}{\thmlistleft}%
      \setlength{\labelwidth}{\thmlistleft}%
      \setlength{\rightmargin}{\thmlistright}%
      \setlength{\partopsep}{\thmlistpartopsep}%
      \setlength{\topsep}{\thmlisttopsep}%
      \setlength{\parsep}{\thmlistparsep}%
      \setlength{\itemsep}{\thmlistitemsep}}}%
  {\end{list}}%

\newcommand{\prtlbl}[1]{{\upshape(#1)}}

\newcounter{rqm}
\newenvironment{rqm}{\begin{list}{\upshape (\arabic{rqm})}%
    {\usecounter{rqm}%
      \setlength{\leftmargin}{\thmlistleft}%
      \setlength{\labelwidth}{\thmlistleft}%
      \setlength{\rightmargin}{\thmlistright}%
      \setlength{\partopsep}{\thmlistpartopsep}%
      \setlength{\topsep}{\thmlisttopsep}%
      \setlength{\parsep}{\thmlistparsep}%
      \setlength{\itemsep}{\thmlistitemsep}}}%
  {\end{list}}%

  {\end{list}}%

  {\end{list}}%

\newenvironment{prf*}[1][Proof]{%
  \begin{proof}[\bf #1]
    \setcounter{equation}{0}
    \renewcommand{\theequation}{\arabic{equation}}}
  {\end{proof}
}

  \newcommand{\proofoftag}[2][:]{(#2)#1}

\newcommand{\pgref}[1]{(\ref{#1})}

\newcommand{\thmref}[2][Theorem~]{#1\pgref{thm:#2}}

\newcommand{\dfnref}[2][Definition~]{#1\pgref{dfn:#2}}

\newcommand{\secref}[2][Section~]{#1\ref{sec:#2}}

\newcommand{\partpgref}[2]{(\ref{#1})\prtlbl{#2}}

\newcommand{\partthmref}[3][Theorem~]{#1\partpgref{thm:#2}{#3}}

\newcommand{\partprpref}[3][Proposition~]{#1\partpgref{prp:#2}{#3}}
\newcommand{\partlemref}[3][Lemma~]{#1\partpgref{lem:#2}{#3}}

\newcommand{\partdfnref}[3][Definition~]{#1\partpgref{dfn:#2}{#3}}

\renewcommand{\eqref}[1]{\pgref{eq:#1}}

\newcommand{\rescite}[2][?]{\cite[#1]{#2}}
\newcommand{\thmcite}[2][?]{\cite[thm.~#1]{#2}}
\newcommand{\corcite}[2][?]{\cite[cor.~#1]{#2}}
\newcommand{\prpcite}[2][?]{\cite[prop.~#1]{#2}}
\newcommand{\lemcite}[2][?]{\cite[lem.~#1]{#2}}
\newcommand{\chpcite}[2][?]{\cite[chap.~#1]{#2}}

\newcommand{\dfncite}[2][?]{\cite[def.~#1]{#2}}
\newcommand{\obscite}[2][?]{\cite[obs.~#1]{#2}}

\makeatletter
\def\@nobreak@#1{\mathchoice%
  {\nobreakdef@\displaystyle\f@size{#1}}%
  {\nobreakdef@\nobreakstyle\tf@size{\firstchoice@false #1}}%
  {\nobreakdef@\nobreakstyle\sf@size{\firstchoice@false #1}}%
  {\nobreakdef@\nobreakstyle\ssf@size{\firstchoice@false #1}}%
  \check@mathfonts}%
\def\nobreakdef@#1#2#3{\hbox{{%
                    \everymath{#1}%
                    \let\f@size#2\selectfont%
                    #3}}}%
\makeatother

\numberwithin{equation}{res}

\newcommand{\s}[1]{\mathsf{#1}}
\newcommand{\Mod}{\mathsf{Mod}}

\newcommand{\ZZ}{\mathbb{Z}}
\newcommand{\QZ}{\mathbb{Q}/\mathbb{Z}}

\newcommand{\Rop}{R^\circ}
\newcommand{\Hom}{\operatorname{Hom}}
\newcommand{\Ext}{\operatorname{Ext}}
\newcommand{\Tor}{\operatorname{Tor}}
\newcommand{\depth}{\operatorname{depth}}
\newcommand{\width}{\operatorname{width}}

\newcommand{\RHom}{\operatorname{\mathbf{R}Hom}}
\newcommand{\LTensor}{\otimes^{\mathbf{L}}}

\def\widebardisplay#1{%
  \setbox0=\hbox{$\displaystyle #1$}
  \dimen0=\wd0%
  \advance\dimen0 by -4pt
  \vbox{%
    \nointerlineskip%
    \moveright 2pt 
    \vbox{\hrule width \dimen0}%
    \nointerlineskip%
    \kern 1pt
    \box0%
    }%
  }

\def\widebartext#1{%
  \setbox0=\hbox{$#1$}
  \dimen0=\wd0%
  \advance\dimen0 by -4pt
  \vbox{%
    \nointerlineskip%
    \moveright 2pt 
    \vbox{\hrule width \dimen0}%
    \nointerlineskip%
    \kern 1pt
    \box0%
    }%
  }

\def\widebarscript#1{%
  \setbox0=\hbox{$\scriptstyle #1$}
  \dimen0=\wd0%
  \advance\dimen0 by -3pt
  \vbox{%
    \nointerlineskip%
    \moveright 1.5pt 
    \vbox{\hrule width \dimen0}%
    \nointerlineskip%
    \kern .8pt
    \box0%
    }%
  }

\def\widebarscriptscript#1{%
  \setbox0=\hbox{$\scriptscriptstyle #1$}
  \dimen0=\wd0%
  \advance\dimen0 by -2pt
  \vbox{%
    \nointerlineskip%
    \moveright 1pt 
    \vbox{\hrule width \dimen0}%
    \nointerlineskip%
    \kern .6pt
    \box0%
    }%
  }

\begin{document}

\title{Cotorsion pairs induced by duality pairs}

\author{Henrik Holm \ } 
\address{Department of Basic Sciences and Environment, Faculty of Life
  Sciences, University of Copenhagen, Thorvaldsensvej 40, DK-1871
  Frederiksberg C, Denmark}
\email{hholm@life.ku.dk} 
\urladdr{http://www.dina.life.ku.dk/\~{}hholm}

\author{ \ Peter J{\o}rgensen} 
\address{School of Mathematics and Statistics, Newcastle University,
  Newcastle upon Tyne NE1 7RU, United Kingdom}
\email{peter.jorgensen@newcastle.ac.uk}
\urladdr{http://www.staff.ncl.ac.uk/peter.jorgensen}






\begin{abstract}
  We introduce the notion of a duality pair and demonstrate how the
  left half of such a pair is ``often'' covering and preenveloping. As
  an application, we generalize a result by Enochs et al.~on Auslander
  and Bass classes, and we prove that the class of Gorenstein
  injective modules---introduced by Enochs and Jenda---is covering
  when the ground ring has a dualizing complex.
\end{abstract}

\maketitle

\section*{Introduction}

What is now known as semi\-dualizing modules were studied more than 25
years ago under other names by e.g.~Foxby \cite{HBF72} (PG-modules of
rank one), Golod \cite{ESG84} (suitable modules) and Vasconcelos
\cite{dtmc} (spherical modules). As a common generalization of the
notion of a semidualizing module and that of a dualizing complex---in
the sense of Hartshorne \cite{rad}---Christensen \cite{LWC01a}
introduced in 2001 the notion of a semidualizing complex,
cf.~\pgref{pg:sc}.

Avramov and Foxby \cite{LLAHBF97} and Christensen \cite{LWC01a}
demonstrated how a semiduali\-zing complex $C$ over a commutative
noetherian ring $R$ gives rise to two important classes of
$R$-modules, namely the so-called Auslander class $\s{A}^C_0$ and Bass
class $\s{B}^C_0$, cf.~\pgref{pg:AB}.  Semi\-dualizing complexes and
their Auslander and Bass classes have caught the attention of several
authors, but this paper is motivated by a result of Enochs et
al.~\cite{EEE-HH}, for which we prove the following generalization in
\thmref{Acotorsion}.
\begin{res*}[Theorem A]
  Let $R$ be a commutative noetherian ring, and let $C$ be a
  semidualizing complex of $R$-modules. Then the following conclusions
  hold:
  \begin{prt}
  \item $(\s{A}^C_0,(\s{A}^C_0)^\perp)$ is a perfect cotorsion pair,
    in particular, the class $\s{A}^C_0$ is covering. Furthermore,
    $\s{A}^C_0$ is preenveloping.
  \item The class $\s{B}^C_0$ is covering and preenveloping.
  \end{prt}
\end{res*}
Cotorsion pairs---introduced by Salce \cite{Salce}---covering classes
and preenveloping classes are central notions in relative homological
algebra. We refer the reader to \pgref{pg:CovEnv} and \pgref{pg:C},
and further to the monograph \cite{rha} by Enochs and Jenda, for
relevant details about these notions.

Theorem A extends the main result of \cite{EEE-HH} in two directions:
In \cite{EEE-HH}, $C$ is assumed to be a semidualizing \textsl{module}
(as opposed to a semidualizing \textsl{complex}), and furthermore the
covering property of $\s{B}^C_0$ is new.

To prove Theorem A, we first establish \thmref{main} and then combine
it with the fact that $\s{A}^C_0$ and $\s{B}^C_0$ are parts of
appropriate \emph{duality pairs}. The latter notion is introduced in
\dfnref{complemented}.  The technique used to prove Theorem A applies
to show that several other classes of modules are covering and/or
preenveloping; see for example Theorems B below.

Now, assume that $R$ has a dualizing complex $D$ in the sense of
Hartshorne~\cite{rad}, and consider a semidualizing
$R$-\textsl{module} $C$. Then \mbox{$C^\dagger=\RHom_R(C,D)$} is a
semidualizing complex for which the associated Auslander and Bass
classes can be characterized in terms of two homological dimensions
introduced in \cite{HHlPJr06}, and studied further by Sather-Wagstaff,
Sharif, and White \cite{SSW-TS-DWb, SSW-TS-DWa, SSW-08, SSW-SY}. More
precisely, for any $R$-module $M$ one has the following equivalences:
\begin{align*}
  M \in \s{A}^{C^\dagger}_0(R) &\iff 
  \operatorname{\text{$C$}-Gfd}_RM\leqslant\operatorname{dim} R, \\
  M \in \s{B}^{C^\dagger}_0(R) &\iff
  \operatorname{\text{$C$}-Gid}_RM\leqslant\operatorname{dim} R.
\end{align*}
Here $\operatorname{\text{$C$}-Gfd}_RM$ and
$\operatorname{\text{$C$}-Gid}_RM$ are the so-called $C$-Gorenstein
flat and $C$-Goren\-stein injective dimensions of $M$. Naturally,
Theorem A applies to the semidualizing complex $C^\dagger$, but in
view of the equivalences above, Theorem B below---which is a special
case of \thmref{GFCcotorsion}---gives more information.

\begin{res*}[Theorem B]
  Let $R$ be a commutative noetherian ring with a dualizing complex,
  and let $n \geqslant 0$ be an integer. Consider the following
  classes of $R$-modules:
  \begin{align*}
    \s{GF}^C_n &= \{ M \,|\,
    \operatorname{\text{$C$}-Gfd}_RM \leqslant n \},
    \\
    \s{GI}^C_n &= \{ M \,|\,
    \operatorname{\text{$C$}-Gid}_RM \leqslant n \}.
  \end{align*}
  Then the next conclusions hold:
  \begin{prt}
  \item \mbox{$(\s{GF}^C_n,(\s{GF}^C_n)^\perp)$} is a perfect
    cotorsion pair, in particular, $\s{GF}^C_n$ is covering.
    Furthermore, $\s{GF}^C_n$ is preenveloping.
  \item The class $\s{GI}^C_n$ is covering and preenveloping.
  \end{prt}
\end{res*}

For \mbox{$C=R$} and \mbox{$n=0$}, Theorem B\prtlbl{a} asserts that
the class of Gorenstein flat modules is the left half of a perfect
cotorsion pair, and that it is preenveloping.  The first of these
results is proved by Enochs and L{\'o}pez-Ramos
\corcite[2.11]{EEnJLR02} and the other follows immediately by
combining \cite[prop.~2.10, thm.~2.5, and~rmk.~3]{EEnJLR02} with
\thmcite[5.7]{CFH-06}.  For \mbox{$C=R$} and \mbox{$n=0$}, the second
part of Theorem B\prtlbl{b} asserts that the class of Gorenstein
injective modules is preenveloping. This is proved in
\corcite[2.7]{EEnJLR02}. Actually, Krause \thmcite[7.12]{HKr05} proves
the existence of special Gorenstein injective preenvelopes.

\section{Preliminaries} \label{sec:pre}

In this section we introduce our terminology and recall a few notions
relevant for this paper.

\begin{stp}
  Throughout, $R$ denotes a ring with identity, and $\Rop$ its
  opposite ring. Unless otherwise specified, all modules under
  consideration are unitary left modules. Recall that a right
  $R$-module can be identified with a left $\Rop$-module. We write
  $\Mod(R)$ for the category of all (left) $R$-modules.
\end{stp}

\begin{bfhpg}[Covers and envelopes]
  \label{pg:CovEnv}
  The following notions were coined by Enochs \cite{EEn81}.
  
  Let $\s{M}$ be any class of $R$-modules.  An $\s{M}$-\emph{precover}
  of an $R$-module $N$ is a homomorphism \mbox{$\varphi \colon M \to
    N$}, where $M$ is in $\s{M}$, with the property that for every
  homomorphism \mbox{$\varphi' \colon M' \to N$}, where $M'$ is in
  $\s{M}$, there exists a (not necessarily unique) homomorphsm
  \mbox{$\psi \colon M' \to M$} with \mbox{$\varphi'=\varphi\psi$}.
  An $\s{M}$-precover \mbox{$\varphi \colon M \to N$} is an
  $\s{M}$-\emph{cover} if every homomorphism \mbox{$\psi \colon M \to
    M$} satisfying \mbox{$\varphi=\varphi\psi$} is an automorphism.
  The class $\s{M}$ is called \emph{(pre)covering} if every $R$-module
  has an $\s{M}$-(pre)cover.

  The notion of an $\s{M}$-\emph{(pre)envelope} is categorically dual
  to that of an $\s{M}$-(pre)cover, and thus we will omit the
  definition here.
\end{bfhpg}

\begin{bfhpg}[Cotorsion pairs]
  \label{pg:C}
  For a class $\s{M}$ of $R$-modules one defines:
  \begin{align*}
    {}^\perp\s{M} &= \{ X \in \Mod(R) \,|\, \Ext^1_R(X,M)=0 \text{ for
      all } M \in \s{M} \}, \\
    \s{M}^\perp &= \{ Y \in \Mod(R) \,|\, \Ext^1_R(M,Y)=0 \text{ for
      all } M \in \s{M} \}.
  \end{align*}
  A \emph{cotorsion pair} is a pair $(\s{M},\s{N})$ of classes of
  $R$-modules with $\s{M}={}^\perp\s{N}$ and $\s{M}^\perp=\s{N}$.  A
  cotorsion pair $(\s{M},\s{N})$ is called \emph{perfect} if $\s{M}$
  is covering and $\s{N}$ is enveloping.  These notions go back to
  Salce \cite{Salce}.
\end{bfhpg}

\begin{bfhpg}[The derived category]
  We denote by $\s{D}(R)$ the \emph{derived category} of the abelian
  category $\Mod(R)$. We write $\s{D}_{\mathrm{b}}(R)$ for the full
  subcategory of $\s{D}(R)$ whose objects have bounded homology. The
  right derived Hom functor and the left derived tensor product
  functor are denoted by $\RHom_R(-,-)$ and $-\LTensor_R-$,
  respectively.  The reader is referred to Weibel \chpcite[10]{Wei}
  for further details.
\end{bfhpg}

\begin{bfhpg}[Semidualizing complexes] \label{pg:sc}
  The following is from Christensen \dfncite[(2.1)]{LWC01a}.

  Assume that $R$ is commutative and noetherian. A complex \mbox{$C
    \in \s{D}_{\mathrm{b}}(R)$} with degreewise finitely generated
  homology is \emph{semidualizing} if the natural homothety morphism
  \mbox{$R \to \RHom_R(C,C)$} is an isomorphism in the derived
  category $\s{D}(R)$.
\end{bfhpg}

\begin{bfhpg}[Auslander and Bass classes] 
  \label{pg:AB} 
  Assume that $R$ is commutative and noetherian, and let $C$ be a
  semidualizing $R$-complex. The following definitions are due to
  Avramov and Foxby \cite[(3.1)]{LLAHBF97} and Christensen
  \dfncite[(4.1)]{LWC01a}.

  The \emph{Auslander class} $\s{A}^C(R)$ consists of all $M \in
  \s{D}_{\mathrm{b}}(R)$ such that $C\LTensor_RM \in
  \s{D}_{\mathrm{b}}(R)$ and the canonical morphism $M \to
  \RHom_R(C,C\LTensor_RM)$ in $\s{D}(R)$ is an isomorphism.

  The \emph{Bass class} $\s{B}^C(R)$ consists of all $N \in
  \s{D}_{\mathrm{b}}(R)$ such that \mbox{$\RHom_R(C,N) \in
  \s{D}_{\mathrm{b}}(R)$} and the canonical morphism
  $C\LTensor_R\RHom_R(C,N)\to N$ in $\s{D}(R)$ is an isomorphism.

  We write $\s{A}^C_0(R)$ and $\s{B}^C_0(R)$---or simply $\s{A}^C_0$
  and $\s{B}^C_0$ if the ground ring is under\-stood---for the class of
  $R$-modules which, when considered as objects in $\s{D}(R)$, belong
  to $\s{A}^C(R)$ and $\s{B}^C(R)$, respectively.
\end{bfhpg}

\begin{rmk}
  If $C$ is a semidualizing $R$-\textsl{module} then it is possible to
  define $\s{A}^C_0(R)$ and $\s{B}^C_0(R)$ directly in $\Mod(R)$
  without using $\s{D}(R)$, cf.~\obscite[(4.10)]{LWC01a}.
\end{rmk}

\begin{bfhpg}[Homological dimensions] \label{pg:homdim}
  Let $M$ be an arbitrary $R$-module.

  We write $\operatorname{fd}_RM$ and $\operatorname{id}_RM$ for the
  \emph{flat} and \emph{injective} dimension of $M$. These classical
  notions go back to Cartan and Eilenberg \cite{CarEil}.

  We write $\operatorname{Gfd}_RM$ and $\operatorname{Gid}_RM$ for the
  \emph{Gorenstein flat} and \emph{Gorenstein injective} dimension of
  $M$. These notions were introduced by Enochs, Jenda et al.
  \cite{EEnOJn95b, EJT-93} and have subsequently been studied by
  several authors.

  When $R$ is commutative and noetherian, the definitions of Enochs,
  Jenda et al.~mentioned above have been extended in \cite{HHlPJr06}:
  For a semi\-dualizing $R$-module $C$, cf.~\pgref{pg:sc},
  \cite{HHlPJr06} introduces a \emph{$C$-Go\-ren\-stein flat}
  dimension $\operatorname{\text{$C$}-Gfd}_RM$ and a
  \emph{$C$-Go\-ren\-stein injective} dimension
  $\operatorname{\text{$C$}-Gid}_RM$. For \mbox{$C=R$} these
  invariants agree with $\operatorname{Gfd}_RM$ and
  $\operatorname{Gid}_RM$, respectively. The $C$-Gorenstein dimensions
  have been studied by e.g.~Sather-Wagstaff, Sharif, and White
  \cite{SSW-TS-DWb, SSW-TS-DWa, SSW-08, SSW-SY}.
\end{bfhpg}

\begin{bfhpg}[Depth and width]
  \label{pg:depth}
  Assume that $(R,\mathfrak{m},k)$ is commutative noetherian local.
  The \emph{depth} of a finitely generated $R$-module $M\neq 0$, that
  is, the length of a maximal $M$-regular sequence, can be computed as
  \begin{displaymath}
    \depth_RM = \inf\{ m \in \ZZ \,|\, \Ext^m_R(k,M) \neq 0\}.
  \end{displaymath}
  Foxby \cite{HBF79} defines the depth of an arbitrary $R$-module $M$
  by the equality above, and Yassemi \cite{SYs98b} studies the dual
  notion of \emph{width}, which is defined by
  \begin{displaymath}
    \width_RM = \inf\{ m \in \ZZ \,|\, \Tor_m^R(k,M) \neq 0\}.
  \end{displaymath}
  Note that $\depth_R0=\width_R0=\infty$.
\end{bfhpg}

\section{Duality pairs} \label{sec:duality}

In this section we define duality pairs and give several examples.  In
the next section we will prove how suitable duality pairs induce
cotorsion pairs. For unexplained notions and notation, the reader is
referred to \secref{pre}.

\begin{dfn}
  \label{dfn:complemented}
  A \emph{duality pair} over $R$ is a pair $(\s{M},\s{C})$, where
  $\s{M}$ is a class of $R$-modules and $\s{C}$ is a class of
  $\Rop$-modules, subject to the following conditions:
  \begin{rqm}

  \item For an $R$--module $M$, one has $M \in \s{M}$ if and only if
    $\Hom_{\ZZ}(M,\QZ) \in \s{C}$.

  \item $\s{C}$ is closed under direct summands and finite direct sums.

  \end{rqm}

  A duality pair $(\s{M},\s{C})$ is called \emph{(co)product-closed} if the
  class $\s{M}$ is closed under (co)products in the category of all
  $R$-modules.

  A duality pair $(\s{M},\s{C})$ is called \emph{perfect} if it is
  coproduct-closed, if $\s{M}$ is closed under extensions, and if $R$
  belongs to $\s{M}$.
\end{dfn}

\begin{exa}
  Consider for each integer \mbox{$n \geqslant 0$} the following
  module classes:
  \begin{align*}
    \s{F}_n &= 
    \{ M \in \Mod(R) \,|\, \operatorname{fd}_RM \leqslant n \}, 
    \\
    \s{I}_n &= 
    \{ M \in \Mod(\Rop) \,|\, \operatorname{id}_{\Rop}M \leqslant n \}.
  \end{align*}
  The following examples of duality pairs are well-known.
  \begin{prt}

  \item $(\s{F}_n,\s{I}_n)$ is a perfect duality pair.  If $R$ is
    right coherent then this pair is product-closed by a classical
    result of Chase \thmcite[2.1]{Chase}.

  \item If $R$ is right noetherian then $(\s{I}_n,\s{F}_n)$ is a
    product- and coproduct-closed duality pair (over $\Rop$), cf.~Xu
    \lemcite[3.1.4]{xu} and Bass \thmcite[1.1]{HBs62}.

  \end{prt}
\end{exa}

\begin{exa}
  Let $\s{B}$ be a class of finitely presented $R$-modules. Following
  Lenzing \cite[\S2]{Lenzing} and \dfncite[2.3]{Holm}, we consider the
  class $\s{M}$ of modules with \emph{support} in $\s{B}$, and the
  class $\s{C}$ of modules with \emph{cosupport} in $\s{B}$ defined by:
  \begin{align*}
    \s{M} &= \varinjlim \s{B},
    \\
    \s{C} &= 
    \operatorname{Prod}\{ \Hom_{\ZZ}(B,\QZ) \,|\, B \in \s{B} \}.
  \end{align*}
  Then $(\s{M},\s{C})$ is a coproduct-closed duality pair by
  \prpcite[2.1]{Lenzing} and \thmcite[1.4]{Holm}. For example, if
  $\s{B}$ is the class of all finitely generated projective
  $R$-modules, then $(\s{M},\s{C}) = (\s{F}_0,\s{I}_0)$ by a classical
  result of Lazard \cite{DLz69}.
\end{exa}

\enlargethispage{2ex}

\begin{prp} 
  \label{prp:ABduality}
  Assume that $R$ is commutative and noetherian, and let $C$ be a
  semidualizing $R$-complex. Then one has:
  \begin{prt}

  \item $(\s{A}^C_0,\s{B}^C_0)$ is a perfect and product-closed
    duality pair.

  \item $(\s{B}^C_0,\s{A}^C_0)$ is a product- and coproduct-closed
    duality pair.

  \end{prt}
\end{prp}

\begin{proof}
  That $(\s{A}^C_0,\s{B}^C_0)$ and $(\s{B}^C_0,\s{A}^C_0)$ are duality
  pairs follow from (the proof of) \lemcite[(3.2.9)]{lnm}.  That
  $\s{A}^C_0$ and $\s{B}^C_0$ are closed under products and coproducts
  follow from (the proof of) \lemcite[5.6]{CFH-06}.  The class
  $\s{A}^C_0$ clearly contains $R$, and it is closed under extensions
  by (the proof of) \lemcite[(3.1.13)]{lnm}.  
\end{proof}

\begin{lem}
  \label{lem:GF-GI}
  Consider for each integer \mbox{$n \geqslant 0$} the following
  module classes:
  \begin{align*}
    \s{GF}_n &= 
    \{ M \in \Mod(R) \,|\, \operatorname{Gfd}_RM \leqslant n \}, 
    \\
    \s{GI}_n &= 
    \{ M \in \Mod(\Rop) \,|\, \operatorname{Gid}_{\Rop}M \leqslant n \}.
  \end{align*}
  Then the following conclusions hold:
  \begin{prt}

  \item If $R$ is right coherent then $(\s{GF}_n,\s{GI}_n)$ is a
    perfect duality pair. If $R$ is com\-mu\-tative noetherian with a
    dualizing complex then this duality pair is pro\-duct-closed.

  \item If $R$ is commutative noetherian with a dualizing complex then
    $(\s{GI}_n,\s{GF}_n)$ is a product- and coproduct-closed duality
    pair (over $\Rop$).
  \end{prt}
\end{lem}

\begin{proof}
  \proofoftag{a} Since $R$ is right coherent, it follows by
  \prpcite[3.11]{HHl04a} that the given pair is a duality pair.  The
  class $\s{GF}_n$ is closed under coproducts by
  \prpcite[3.13]{HHl04a}, and to see that $\s{GF}_n$ is closed under
  extensions one applies \thmcite[3.14 and 3.15]{HHl04a}. It is clear
  that $R$ belongs to $\s{GF}_n$.

  If $R$ is commutative and noetherian with a dualizing complex then
  $\s{GF}_n$ is closed under products by \thmcite[5.7]{CFH-06}.

  \proofoftag{b} Since $R$ is commutative with a dualizing complex, it
  follows by (the proof of) \prpcite[5.1]{CFH-06} that the given pair
  is a duality pair. The class $\s{GI}_n$ is closed under products by
  \thmcite[2.6]{HHl04a}, and it is closed under coproducts by
  \thmcite[6.9]{CFH-06}.
\end{proof}

\begin{prp}
  \label{prp:GFC}
  Assume that $R$ is commutative and noetherian, and let $C$ be a
  semidualizing $R$-module. Consider for each \mbox{$n \geqslant 0$}
  the following module classes:
  \begin{align*}
    \s{GF}^C_n &= \{ M \in \Mod(R) \,|\,
    \operatorname{\text{$C$}-Gfd}_RM \leqslant n \},
    \\
    \s{GI}^C_n &= \{ M \in \Mod(R) \,|\,
    \operatorname{\text{$C$}-Gid}_RM \leqslant n \}.
  \end{align*}
  Then one has the following conclusions:
  \begin{prt}

  \item $(\s{GF}^C_n,\s{GI}^C_n)$ is a perfect duality pair. If $R$
    has a dualizing complex then this duality pair is product-closed.

  \item If $R$ has a dualizing complex then $(\s{GI}^C_n,\s{GF}^C_n)$
    is a product- and coproduct-closed duality pair.

  \end{prt}
\end{prp}

\begin{proof}
  \proofoftag{a} We denote by $R\ltimes C$ the trivial extension of
  $R$ by $C$, cf.~\cite[\S 3.3]{BruHer}. For any $R$-module $M$, it
  follows by \thmcite[2.16]{HHlPJr06} that
  \begin{displaymath}
    \tag{\text{$\dagger$}}
    \operatorname{\text{$C$}-Gfd}_RM = 
    \operatorname{Gfd}_{R\ltimes C}M
    \qquad \text{and} \qquad
    \operatorname{\text{$C$}-Gid}_RM = 
    \operatorname{Gid}_{R\ltimes C}M.
  \end{displaymath}
  Combining this with \prpcite[3.11]{HHl04a} one gets that
  \begin{align*}
    \operatorname{\text{$C$}-Gid}_R\Hom_{\ZZ}(M,\QZ) 
    &= 
    \operatorname{Gid}_{R\ltimes C}\Hom_{\ZZ}(M,\QZ) \\
    &= 
    \operatorname{Gfd}_{R\ltimes C}M \\
    &= 
    \operatorname{\text{$C$}-Gfd}_RM,
  \end{align*}
  from which we conclude that $(\s{GF}^C_n,\s{GI}^C_n)$ is a duality
  pair.  By \partlemref{GF-GI}{a}, the class \mbox{$\s{GF}_n(R\ltimes
    C)$} is closed under coproducts and extensions; and combining this
  with the first equality in $(\dagger)$, it follows that $\s{GF}^C_n$
  is closed under coproducts and extensions as well. Also note that
  $R$ belongs to $\s{GF}^C_n$ by \rescite[exa.~2.8(c)]{HHlPJr06}.

  If $R$ has a dualizing complex then so has \mbox{$R\ltimes C$},
  since it is a module finite extension of $R$. Hence
  \mbox{$\s{GF}_n(R\ltimes C)$} is closed under products by
  \partlemref{GF-GI}{a}, and by the first equality in $(\dagger)$ we
  then conclude that $\s{GF}^C_n$ is closed under products.

  \proofoftag{b} Similar to the proof of part \prtlbl{a}, but using
  the second equality in $(\dagger)$ instead of the first, and using
  that for any $R$-module $M$ one has:
  \begin{align*}
    \operatorname{\text{$C$}-Gfd}_R\Hom_{\ZZ}(M,\QZ) 
    &= 
    \operatorname{Gfd}_{R\ltimes C}\Hom_{\ZZ}(M,\QZ) \\
    &= 
    \operatorname{Gid}_{R\ltimes C}M \\
    &= 
    \operatorname{\text{$C$}-Gid}_RM.
  \end{align*}
  The equalities above follow from $(\dagger)$ and (the proof of)
  \prpcite[5.1]{CFH-06}.
\end{proof}

\begin{prp}
  Let $(R,\mathfrak{m},k)$ be commutative noetherian local. Consider
  for each integer \mbox{$n \geqslant 0$} the following module
  classes:
  \begin{align*}
    \s{D}_n &= 
    \{ M \in \Mod(R) \,|\, \depth_RM \geqslant n \}, 
    \\
    \s{W}_n &= 
    \{ M \in \Mod(R) \,|\, \width_RM \geqslant n \}. 
  \end{align*}
  Then the following conclusions hold:
  \begin{prt}
  \item $(\s{D}_n,\s{W}_n)$ is a product- and coproduct-closed duality
    pair. If \mbox{$n \leqslant \depth R$} then this duality pair is
    perfect.
  \item $(\s{W}_n,\s{D}_n)$ is a product- and coproduct-closed duality
    pair.
  \end{prt}
\end{prp}

\begin{proof}
  For every $R$-module $M$ one has:
  \begin{align*}
    \depth_R\Hom_{\ZZ}(M,\QZ)&=\width_RM, \\
    \width_R\Hom_{\ZZ}(M,\QZ)&=\depth_RM,
  \end{align*}
  from which it follows that the pairs in \prtlbl{a} and \prtlbl{b}
  are duality pairs. It is trivial from the definitions of depth and
  width, cf.~\pgref{pg:depth}, that $\s{D}_n$ and $\s{W}_n$ are closed
  under products, coproducts, direct summands and extensions.
\end{proof}

We end this section by noting that the next easily proved result can
be applied to construct new duality pairs from existing ones.

\begin{prp}
  Let $(\s{M}_\mu,\s{C}_\mu)$ be a family of duality pairs over $R$.
  Then their inter\-sec\-tion $(\bigcap \s{M}_\mu,\bigcap \s{C}_\mu)$
  is also a duality pair.  Furthermore, the following hold:
  \begin{prt}
  \item If each $(\s{M}_\mu,\s{C}_\mu)$ is (co)product-closed then so is
    $(\bigcap \s{M}_\mu,\bigcap \s{C}_\mu)$.
  \item If each $(\s{M}_\mu,\s{C}_\mu)$ is perfect then so is $(\bigcap
    \s{M}_\mu,\bigcap \s{C}_\mu)$. \qed
  \end{prt}
\end{prp}

\enlargethispage{2ex}

\section{Existence of preenvelopes, covers, and cotorsion pairs}

The main result of this section, \thmref{main}, shows that the left
half of a duality pair is ``often'' preenveloping and covering.  We
apply this result to a few of the duality pairs found in
\secref{duality}.

\begin{thm} \label{thm:main}
  Let $(\s{M},\s{C})$ be a duality pair. Then $\s{M}$ is closed under pure
  sub\-modules, pure quotients, and pure extensions. Furthermore, the
  following hold:
  \begin{prt}
  \item If $(\s{M},\s{C})$ is product-closed then $\s{M}$ is preenveloping.
  \item If $(\s{M},\s{C})$ is coproduct-closed then $\s{M}$ is covering.
  \item If $(\s{M},\s{C})$ is perfect then $(\s{M},\s{M}^\perp)$ is a perfect
    cotorsion pair.
  \end{prt}
\end{thm}

\begin{proof}
  First we prove that $\s{M}$ is closed under pure sub\-modules, pure
  quotients, and pure extensions, that is, given a pure exact sequence
  of $R$-modules,
  \begin{displaymath}
    0 \longrightarrow M' \longrightarrow M \longrightarrow M''
    \longrightarrow 0, 
  \end{displaymath}
  then $M$ is in $\s{M}$ if and only if $M', M''$ are in $\s{M}$.  Applying
  $\Hom_{\ZZ}(-,\QZ)$ to the sequence above, we get by Jensen and
  Lenzing \thmcite[6.4]{CUJ-HL-book} a split exact sequence,
  \begin{displaymath}
    0 \longrightarrow \Hom_{\ZZ}(M'',\QZ) \longrightarrow 
    \Hom_{\ZZ}(M,\QZ) \longrightarrow  
    \Hom_{\ZZ}(M',\QZ) \longrightarrow 0.
  \end{displaymath}
  By 
  \partdfnref[]{complemented}{2} it follows that $\Hom_{\ZZ}(M,\QZ)$
  is in $\s{C}$ if and only if $\Hom_{\ZZ}(M',\QZ)$ and
  $\Hom_{\ZZ}(M'',\QZ)$ both are in $\s{C}$.  The desired conclusion now
  follows by \partdfnref[]{complemented}{1}.

  \proofoftag{a} We have proved that $\s{M}$ is closed under pure
  submodules. Since $\s{M}$ is also closed under products by assumption,
  it follows by Rada and Saor{\'i}n \corcite[3.5(c)]{Rada-Saorin} that
  $\s{M}$ is preenveloping.

  \proofoftag{b} We have proved that $\s{M}$ is closed under pure
  quotients. By assumption, $\s{M}$ is also closed under coproducts,
  and therefore it follows by \thmcite[2.5]{Holm-Jorgensen} that
  $\s{M}$ is covering.

  \proofoftag{c} We have proved that $\s{M}$ is closed under pure
  submodules and pure quotients. By assumption, $\s{M}$ is also closed
  under coproducts and extensions, and $R$ belongs to $\s{M}$.  Thus
  \thmcite[3.4]{Holm-Jorgensen} implies that $(\s{M},\s{M}^\perp)$ is a
  perfect cotorsion pair.
\end{proof}

As mentioned in the introduction, in the case where $C$ is a
semidualizing \textsl{module} (as opposed to as semidualizing
\textsl{complex}), the following result---except the first assertion
in part (b)---is proved by Enochs et al.~\cite{EEE-HH}.

\begin{thm}
  \label{thm:Acotorsion}
  Assume that $R$ is commutative and noetherian, and let $C$ be a
  semidualizing $R$-complex. Then the following conclusions hold:
  \begin{prt}
  \item $(\s{A}^C_0,(\s{A}^C_0)^\perp)$ is a perfect cotorsion pair,
    in particular, the class $\s{A}^C_0$ is covering. Furthermore,
    $\s{A}^C_0$ is preenveloping.
  \item The class $\s{B}^C_0$ is covering and preenveloping.
  \end{prt}
\end{thm}

\begin{proof}
  \proofoftag{a} By \partprpref{ABduality}{a}, the class $\s{A}^C_0$
  is the left half of a perfect and product-closed duality pair. Thus
  the conclusions follow from \partthmref{main}{c,a}.

  \proofoftag{b} By \partprpref{ABduality}{b}, the class $\s{B}^C_0$
  is the left half of a product- and coproduct-closed duality pair.
  The conclusions follow from \partthmref{main}{a,b}.
\end{proof}

\begin{thm}
  \label{thm:GFCcotorsion}
  Assume that $R$ is commutative and noetherian, let $C$ be a
  semi\-dual\-izing $R$-module, and let $n \geqslant 0$ be an integer.
  Then one has:
  \begin{prt}
  \item $(\s{GF}^C_n,(\s{GF}^C_n)^\perp)$ is a perfect cotorsion pair,
    in particular, $\s{GF}^C_n$ is covering. If, in addition, $R$ has a
    dualizing complex then $\s{GF}^C_n$ is preenveloping.
  \item If $R$ has a dualizing complex then $\s{GI}^C_n$ is covering
    and preenveloping.
  \end{prt}
\end{thm}

\begin{proof}
  \proofoftag{a} By \partprpref{GFC}{a}, the class $\s{GF}^C_n$ is the
  left half of a perfect duality pair. Thus the claimed perfect
  cotorsion pair exists by \partthmref{main}{c}. Under the assumption
  of the existence of a dualizing $R$-complex, $\s{GF}^C_n$ is also
  product-closed by
  \partprpref{GFC}{a}, and therefore $\s{GF}^C_n$ is preenveloping by
  \partthmref{main}{a}.

  \proofoftag{b} If $R$ has a dualizing complex,
  \partprpref{GFC}{b} gives that $\s{GI}^C_n$ is the left half of a
  product- and coproduct-closed duality pair. Thus the assertions
  follow from \partthmref{main}{b,a}.
\end{proof}

\def\cprime{$'$}
  \newcommand{\arxiv}[2][AC]{\mbox{\href{http://arxiv.org/abs/#2}{\sf arXiv:#2
  [math.#1]}}}
  \newcommand{\oldarxiv}[2][AC]{\mbox{\href{http://arxiv.org/abs/math/#2}{\sf
  arXiv:math/#2
  [math.#1]}}}\providecommand{\MR}[1]{\mbox{\href{http://www.ams.org/mathscine%
t-getitem?mr=#1}{#1}}}
  \renewcommand{\MR}[1]{\mbox{\href{http://www.ams.org/mathscinet-getitem?mr=#%
1}{#1}}}
\providecommand{\bysame}{\leavevmode\hbox to3em{\hrulefill}\thinspace}
\providecommand{\MR}{\relax\ifhmode\unskip\space\fi MR }
\providecommand{\MRhref}[2]{%
  \href{http://www.ams.org/mathscinet-getitem?mr=#1}{#2}
}
\providecommand{\href}[2]{#2}


\end{document}